\documentclass[12pt,a4paper, reqno]{amsart}
\usepackage[a4paper]{geometry}
\usepackage{amsthm}
\usepackage{amssymb}
\usepackage{amsmath}
\usepackage{mathtools}
\usepackage[utf8]{inputenc}
\usepackage{comment}

\usepackage[T1]{fontenc}

\usepackage{hyperref}

\usepackage{wasysym}
\usepackage{color}
\usepackage{enumitem}

\usepackage{tikz}
\usetikzlibrary{positioning}
\usetikzlibrary{decorations.pathreplacing}
\usetikzlibrary{patterns}
\pgfdeclarepatternformonly[\LineSpace]{my north east lines}{\pgfqpoint{-1pt}{-1pt}}{\pgfqpoint{\LineSpace}{\LineSpace}}{\pgfqpoint{\LineSpace}{\LineSpace}}%
{
    \pgfsetlinewidth{0.4pt}
    \pgfpathmoveto{\pgfqpoint{0pt}{0pt}}
    \pgfpathlineto{\pgfqpoint{\LineSpace + 0.1pt}{\LineSpace + 0.1pt}}
    \pgfusepath{stroke}
}

\newdimen\LineSpace
\tikzset{
    line space/.code={\LineSpace=#1},
    line space=3pt
}

\newtheorem{thm}{Theorem}[section]

\newtheorem{lem}[thm]{Lemma}
\newtheorem{prop}[thm]{Proposition}

\newtheorem{prob}{\sc Problem}
\theoremstyle{definition}

\theoremstyle{remark}
\newtheorem{rem}[thm]{Remark}

\numberwithin{equation}{section}

\DeclareMathOperator{\Lip}{Lip}

\DeclareMathOperator{\conv}{conv}

\DeclareMathOperator{\Mol}{\mathcal{M}}
\DeclareMathOperator{\N}{\mathbb{N}}
\DeclareMathOperator{\image}{Im}

\title{weakly almost square Lipschitz-free spaces}
\author{Jaan Kristjan Kaasik and Triinu Veeorg}
\address{Institute of Mathematics and Statistics, University of Tartu, Narva~mnt 18, 51009, Tartu, Estonia}
\email{jaan.kristjan.kaasik@ut.ee, triinu.veeorg@ut.ee}
\subjclass{Primary 46B04; Secondary 46B20}
\keywords{Lipschitz-free spaces,  length metric space, weakly almost square, diameter 2 property.}

\begin{document}
\begin{abstract}
   We construct a Lipschitz-free space that is locally almost square but not weakly almost square; this is the first example of such a Banach space. We also prove a result, which indicates that geodesic metric spaces are a potential metric characterization for weakly almost square Lipschitz-free spaces. Lastly, we prove that a Lipschitz-free space can not have the symmetric strong diameter $2$ property.

\end{abstract}

\maketitle

\section{Introduction}
Let $M$ be a metric space with metric $d$ and a fixed point 0. We denote by $\Lip_0(M)$ the Banach space of all Lipschitz functions $f\colon M\rightarrow\mathbb{R}$ with $f(0)=0$ equipped with the obvious linear structure and the norm
\[\|f\|:=\sup\Big\{\frac{|f(x)-f(y)|}{d(x,y)}\colon x, y\in M, x\neq y\Big\}.\]
Let $\delta\colon M\rightarrow \Lip_0(M)^*$ be the canonical isometric embedding of $M$ into $\Lip_0(M)^*$, which is given by $x\mapsto \delta_x$ where $\delta_x(f)=f(x)$. 
The norm closed linear span of $\delta(M)$ in $\Lip_0(M)^*$ is called the \emph{Lipschitz-free space over $M$} and is denoted by $\mathcal{F}(M)$ (see \cite{GOD} and \cite{Weaver} for the background). It is well known that
$\mathcal{F}(M)^*= \Lip_0(M).$

In this article, we study the following properties on Lipschitz-free spaces. A Banach space $X$ is said to be 
\begin{enumerate}
    \item \emph{locally almost square} (briefly, LASQ) if, for every $x\in S_X$, there exists a sequence $(y_i)\subseteq B_X$ such that $\|x\pm y_i\|\rightarrow 1$ and $\| y_i\|\rightarrow 1$;
    \item \emph{weakly almost square} (briefly, WASQ) if, for every $x\in S_X$, there exists a sequence $(y_i)\subseteq B_X$ such that $\|x\pm y_i\|\rightarrow 1$, $\| y_i\|\rightarrow 1$, and $y_i\rightarrow0$ weakly;
    \item \emph{almost square} (briefly, ASQ) if, for every finite subset $(x_j)_{j=1}^n\subseteq S_X$, there exists a sequence $(y_i)\subseteq B_X$ such that $\|x_j\pm y_i\|\rightarrow 1$ for every $j\in\{1,\ldots,n\}$ and $\| y_i\|\rightarrow1$.
\end{enumerate}
It is known that ASQ implies WASQ, and WASQ implies LASQ (see \cite{ALL}). Furthermore, WASQ and ASQ are different properties already in Lipschitz-free spaces, for example $\mathcal{F}([0,1])=L_1([0,1])$ is WASQ but not ASQ. Recall that several diameter $2$ properties (for example the Daugavet property, slice-diameter $2$ property, strong diameter $2$ property) coincide in Lipschitz-free spaces (see \cite{AM},\cite{GPR},\cite{IKW}). Recently it was shown, that a Lipschitz-free space is LASQ if and only if it has any of the aforementioned diameter $2$ properties (see {\cite[Theorem~3.1]{HKO}}). It remained an open question whether WASQ coincides with LASQ in Lipschitz free-spaces. Furthermore, it is unknown whether these two properties are different in general. In Section 2, we provide an example of a Lipschitz-free space that is LASQ, but is not WASQ, thus answering  \cite[Question~3.4]{ALL}. This raises a natural question -- how to describe a class of Banach spaces, where LASQ and WASQ coincide.

By {\cite[Theorem~1.5]{AM}} and {\cite[Theorem~3.1]{HKO}} we know, that a Lipschitz-free space $\mathcal{F}(M)$ is LASQ or has specific diameter $2$ properties if and only if the underlying metric space $M$ is a length space. Recall, that a metric space $M$ is called a \emph{length} space if, for every pair of distinct points $u,v\in M$, the distance $d(u,v)$ is equal to the infimum of the length of rectifiable paths joining them. If the latter infimum is always attained, then $M$ is called a \emph{geodesic} space. Since  WASQ differs from LASQ (and from other diameter $2$ properties) in Lipschitz-free spaces, we start looking for a metric characterization of weakly almost square Lipschitz-free spaces. In Section 3, we prove a result which suggests that any Lipschitz-free space over a geodesic metric space might be WASQ. However, this is a partial result and the full characterization remains as an open problem.

Recall that if a Banach space is ASQ, then it also has the symmetric strong diameter $2$ property (in short, SSD$2$P). It has been proven, that a Lipschitz-free space is never ASQ (see {\cite[Theorem~4.1]{HKO}}), thus arising a question whether there exists a Lipschitz-free space with SSD$2$P. In Section 4, we show that a Lipschitz-free space cannot have the SSD$2$P. 

\subsection*{Notations} For a Banach space $X$ we will denote the closed unit ball by $B_X$, the unit sphere by $S_X$ and the dual space by $X^*$. 

An element in $\mathcal{F}(M)$ of the form 
\[m_{uv}:=\frac{\delta_u-\delta_v}{d(u,v)}\]
for $u,v\in M$ with $u\neq v$ is called a \emph{molecule}. It is well known that the convex hull of all molecules of $\mathcal{F}(M)$ is dense in $B_{\mathcal{F}(M)}$.

Recall that a function $\gamma\colon[\alpha,\beta]\rightarrow M$ is called a \emph{path from $p$ to $q$}, when $\gamma$ is continuous and $\gamma(\alpha)=p$ and $\gamma(\beta)=q$. We denote a path from $p$ to $q$ usually by $\gamma_{pq}$. For a rectifiable path $\gamma\colon [\alpha,\beta]\to M$ we denote its length by $L(\gamma)$, i.e.,
\[
L(\gamma) = \sup\bigl\{\sum_{i=0}^{n-1} d(\gamma(t_i), \gamma(t_{i+1})) \colon n\in \mathbb{N},\, \alpha=t_0\leq \dotsb\leq t_n=\beta \bigr\}.
\]
By default we assume that the domain interval for a path is $[0,1]$. 

Given a point $p$ in $M$ and $r>0$, we denote by $B(p,r)$ the open ball in $M$ centred at $p$ of radius $r$.

% In section 2, we show that LASQ and WASQ are in fact different properties. A Lipschitz-free space $\mathcal{F}(M)$ is LASQ if and only if $M$ is length (see {\cite[Theorem~3.1]{HKO}}). We will construct such a metric space $M$ that $M$ is length, but $\mathcal{F}(M)$ does not have WASQ.

% In section 3, we prove that a Lipschitz-free space can not have SSD2P, thus answering a question by ... (reference? your paper?). 

 \section{Lipschitz-free space with LASQ and without WASQ}
 The purpose of this section is to introduce a Lipschitz-free space that is LASQ, but not WASQ. To do this we take the metric space from {\cite[Example~3.1]{Vee}} and add just enough points to make it a length space, but not a geodesic space.

Consider
\[M:= [0,1]\times (0,1/2] \cup\{(0,0),(1,0)\}\]
with the metric
\[d\big((a,b),(c,d)\big):=\begin{cases} 
    |a-c|,&\text{if }b=d,\\ 
    \min\{a+c,2-a-c\}+|b-d|, &\text{if }b\neq d. \end{cases}\]
Let $(0,0)$ be the fixed point 0. Clearly $M$ is length, thus $\mathcal{F}(M)$ is LASQ by {\cite[Theorem~3.1]{HKO}}.
In the following we will show that $\mathcal{F}(M)$ is not WASQ. 

Set $x:=(0,0)$, $y:=(1,0)$ and let $\pi_1,\pi_2\colon M\rightarrow \mathbb{R}$ be projections defined by
\[\pi_1(a,b)=a\quad \text{ and } \quad \pi_2(a,b)=b \quad \text{ for all }(a,b)\in M.\]
Clearly $\pi_1,\pi_2\in S_{\Lip_0(M)}$. For $\delta,\varepsilon>0$, let
\[\Mol_{\delta,\varepsilon}(M)=\big\{m_{uv}\colon u,v\in M, u\neq v, \delta<\pi_2(u)<\varepsilon, \delta<\pi_2(v)<\varepsilon\big\}.\]
For $u,v\in M$ we denote the segment between points $u$ and $v$ by $[u,v]$, i.e
\[[u,v]=\big\{p\in M\colon d(u,p)+d(p,v)=d(u,v)\big\}.\]

\begin{lem}\label{lemma_something}
Let $\varepsilon,\delta>0$, and let $\nu\in \conv \big(\Mol_{\delta,1}(M)\big)$. If there exists $g\in B_{\Lip_0(M)}$ such that $g(\nu)=\|\nu\|$ and $|g(m_{xy})|<\varepsilon$, then there exists $f\in B_{\Lip_0(M)}$ such that $f(\nu)>(\|\nu\|-\varepsilon-2\delta)/(1+\varepsilon+2\delta)$ and
$f(p)=0$ for every $p\in M$ with $\pi_2(p)\le\delta$.
\end{lem}

\begin{proof}
Let $g\in S_{\Lip_0(M)}$ be such that $g(\nu)=\|\nu\|$ and $|g(m_{xy})|<\varepsilon$. 
Let $u=(0,\delta)$ and $v=(1,\delta)$. By {\cite[Lemma~1.2]{Vee}} we get
\[\|m_{xy}-m_{uv}\|=\frac{d(x,u)+d(v,y)+|d(x,y)-d(u,v)|}{\max\big\{d(x,y),d(u,v)\big\}}=2\delta,\]
and thus
\[|g(m_{uv})|\le|g(m_{xy})|+\|m_{xy}- m_{uv}\|<\varepsilon+2\delta.\]
Let $h=g+g(m_{uv})\pi_1$. Then  $h(u)=h(v)$,
\[\|h\|\le \|g\|+|g(m_{uv})|\|\pi_1\|< 1+\varepsilon+2\delta,\]
and
\[h(\nu)\ge g(\nu)-\|g(m_{uv})\pi_1\|=\|\nu\|-|g(m_{uv})|>\|\nu\|-\varepsilon-2\delta.\]
Let  $f\colon M\rightarrow \mathbb{R}$ be defined by
\[f(p)=\begin{cases} 
    0, & \text{if }\pi_2(p)\le\delta,\\ 
    \big(h(p)-h(u)\big)/\|h\|, &\text{if } \pi_2(p)> \delta. \end{cases}\]
Since $f(u)=f(v)=0$ and $[p,q]\cap\{u,v\}\neq \emptyset$ for all $p,q\in M$ with $\pi_2(p)\le \delta$ and $\pi_2(q)>\delta$, we get  $f\in B_{\Lip_0(M)}$. Clearly  $f(p)=0$ for every $p\in M$ with $\pi_2(p)\le\delta$. Remind that $\nu\in \conv \big(\Mol_{\delta,1}(M)\big)$, thus we have
$\nu=\sum_{i=1}^n\lambda_im_{u_iv_i}$ for some $n\in\mathbb{N}$, $\lambda_i>0$, $u_i,v_i\in M$ with $u_i\neq v_i$ for every $i\in \{1,\ldots,n\}$. Therefore
\begin{align*}
    f(\nu)&=\sum_{i=1}^n\lambda_if(m_{u_iv_i})=\frac{1}{\|h\|}\sum_{i=1}^n\lambda_i\frac{h(u_i)-h(u)-\big(h(v_i)-h(u)\big)}{d(u_i,v_i)}\\
    &=\frac{1}{\|h\|}\sum_{i=1}^n\lambda_ih(m_{u_iv_i})=\frac{h(\nu)}{\|h\|}>\frac{\|\nu\|-\varepsilon-2\delta}{1+\varepsilon+2\delta}.
\end{align*}
\end{proof}

\begin{lem}\label{lemma1}
For every $\varepsilon\in(0,1)$ and for every $\mu\in B_{\mathcal{F}(M)}$, if 
\[\|m_{xy}\pm \mu\|<1+\varepsilon^{2}/64\quad\text{ and }\quad \|\mu\|> 1-\varepsilon^{2}/64,\]
then there exist $\delta>0$ and $\nu\in\conv \big(\Mol_{\delta,\varepsilon}(M)\big)$ such that $\|\mu-\nu\|<\varepsilon$.
\end{lem}

\begin{proof}
Fix $\varepsilon$ and $\mu\in B_{\mathcal{F}(M)}$ such that 
\[\|m_{xy}\pm \mu\|<1+\varepsilon^{2}/64\quad\text{ and }\quad \|\mu\|> 1-\varepsilon^{2}/64.\]
Let $\nu_0=\sum_{i\in I}\lambda_i m_{u_iv_i}\in B_{\mathcal{F}(M)}$, where $I$ is finite, $\lambda_i>0$ and $\sum_{i\in I}\lambda_i=\|\nu_0\|$, be such that $\|\mu-\nu_0\|<\varepsilon^{2}/64$. We may additionally assume that $x,y\notin\{u_i,v_i\colon i\in I\}$ and either $\pi_2(u_i),\pi_2(v_i)\ge \varepsilon/16$ or $\pi_2(u_i),\pi_2(v_i)<\varepsilon/16$ for every $i\in I$. Let $\delta$ be such that $\pi_2(u_i),\pi_2(v_i)>\delta$ for every $i\in I$ and let
\[J=\{i\in I\colon \pi_2(u_i),\pi_2(v_i)\ge \varepsilon/16\}.\]
If $\sum_{i\in J}\lambda_i\le15\varepsilon/16$, then let $\nu=\sum_{i\in I\setminus J}\lambda_i m_{u_iv_i}$. Then  for every $i\in I\setminus J$ we have $\delta<\pi_2(u_i)<\varepsilon/16$ and $\delta<\pi_2(v_i)<\varepsilon/16$ and thus $\nu\in \conv \big(\Mol_{\delta,\varepsilon}(M)\big)$. Additionally
\[\|\mu-\nu\|\le \|\mu-\nu_0\|+\|\nu_0-\nu\|<\frac{\varepsilon^{2}}{64}+\sum_{i\in J}\lambda_i\le \frac{\varepsilon^{2}}{64}+\frac{15}{16}\varepsilon<\varepsilon.\]

Now assume that $\sum_{i\in J}\lambda_i>15\varepsilon/16$. We will show that this is a contradiction. Let $\nu_1=\sum_{i\in J}\lambda_i m_{u_iv_i}$ and
let $g\in S_{\Lip_0(M)}$ be such that $g(\nu_0)=\|\nu_0\|$. Then $g(\nu_1)=\|\nu_1\|$ and
\[g(\mu)\ge g(\nu_0)-\|\mu-\nu_0\|>\|\nu_0\|- \frac{\varepsilon^{2}}{64}\ge \|\mu\|-\|\mu-\nu_0\|{\color{blue} -}\frac{\varepsilon^{2}}{64}>1-\frac{3\varepsilon^{2}}{64},\]
thus 
\[\pm g(m_{xy})\le \|m_{xy}\pm \mu\|-g(\mu)<1+\frac{\varepsilon^{2}}{64}-1+\frac{3\varepsilon^{2}}{64}=\frac{\varepsilon^{2}}{16}.\]
By Lemma \ref{lemma_something} (taking $\delta=\varepsilon/16$, $\varepsilon=\varepsilon/16$, and $\nu=\nu_1$), there exists $f\in B_{\Lip_0(M)}$ such that 
$f(\nu_1)>(\|\nu_1\|-3\varepsilon/16)/(1+3\varepsilon/16)$
and
$f(p)=0$ for every $p\in M$ with $\pi_2(p)<\varepsilon/16$. Then
\[f(\nu_1)>\frac{\|\nu_1\|-3\varepsilon/16}{1+3\varepsilon/16}=\frac{16\sum_{i\in J}\lambda_i-3\varepsilon}{16+3\varepsilon}>\frac{15\varepsilon-3\varepsilon}{16+3\varepsilon}>\frac{12\varepsilon}{19}>\frac{9\varepsilon}{16}.\]
Let $h=f_{xy}+\varepsilon/9\cdot f$,
where
\[f_{xy}(p)=\frac{d(x,y)}{2}\cdot\frac{d(y,p)-d(x,p)}{d(x,p)+d(y,p)}.\]
By {\cite[Lemma~3.6]{GPR}} we get $\|f_{xy}\|=1$ and for all  $p,q\in M$ with $\pi_2(p)\ge \varepsilon/16$ we have $|f_{xy}(p)-f_{xy}(q)|\le d(p,q)/(1+\varepsilon/8)$, since
\[d(x,p)+d(y,p)=\pi_1(p)+\pi_2(p)+1-\pi_1(p)+\pi_2(p)\ge 1+\varepsilon/8=(1+\varepsilon/8)d(x,y).\]

Thus
\begin{align*}
    |h(p)-h(q)|&\le \big|f_{xy}(p)-f_{xy}(q)\big|+\frac{\varepsilon}{9}\big|f(p)-f(q)\big|\\
    &\le \frac{1}{1+\varepsilon/8}d(p,q)+\frac{\varepsilon}{9}d(p,q)\\
    &= \Big(1-\frac{\varepsilon}{8+\varepsilon}\Big)d(p,q)+\frac{\varepsilon}{9}d(p,q)\\
    &< d(p,q)
\end{align*}
for all $p,q\in M$ with $\pi_2(p)\ge \varepsilon/16$. Furthermore $f(p)=0$ when $\pi_2(p)<\varepsilon/16$.
This gives us $\|h\|\le1$. Also $|f_{xy}(\mu)|<\varepsilon^{2}/64$, since $f_{xy}(m_{xy})=1$ and $\|m_{xy}\pm\mu\|< 1+\varepsilon^{2}/64$. Therefore
\[|f_{xy}(\nu_0)|\le|f_{xy}(\mu)|+\|\nu_0-\mu\|<\varepsilon^{2}/32.\]
Furthermore, $h(m_{xy})=1$ and
\[h(\mu)\ge h(\nu_0)-\frac{\varepsilon^{2}}{64}=f_{xy}(\nu_0)+\frac{\varepsilon}{9}f(\nu_1)-\frac{\varepsilon^{2}}{64}>-\frac{\varepsilon^{2}}{32}+\frac{\varepsilon}{9}\frac{9\varepsilon}{16}-\frac{\varepsilon^{2}}{64}=\frac{\varepsilon^{2}}{64},\]
which is a contradiction with $\|m_{xy}+\mu\|< 1+\varepsilon^{2}/64$.

\end{proof}

\begin{thm}\label{notWASQ}
The space $\mathcal{F}(M)$ is not WASQ.
\end{thm}

\begin{proof}
Let $(\mu_i)$ be a sequence in $B_{\mathcal{F}(M)}$ such that 
$\|m_{xy}\pm \mu_i\|\rightarrow 1$ and $\|\mu_i\|\rightarrow 1$. To prove that $\mathcal{F}(M)$ is not WASQ it suffices to show that $(\mu_i)$ is not weakly null, i.e., it suffices to show there exist a subsequence  $(\mu_{k_i})$ of $(\mu_i)$ and $f\in\Lip_0(M)$ such that $(f(\mu_{k_i}))$ does not converge to 0. 
We will construct sequences $(k_i)$, $(\nu_{i})$, and $(\delta_i)$ inductively.  Let $\delta_1=1/2$ and let $k_1\in \mathbb{N}$ be such that \[\|m_{xy}\pm \mu_{k_1}\|<1+\delta_1^2/64\quad\text{ and }\quad \|\mu_{k_1}\|> 1-\delta_1^2/64.\]
By Lemma \ref{lemma1} there exist $\delta_2>0$ and $\nu_1\in\conv \big(\Mol_{2\delta_2,\delta_1}(M)\big)$ such that \\
\[\|\mu_{k_1}-\nu_1\|<\delta_1.\] 

Assume that we have found $k_1,\ldots,k_{i-1}$, $\nu_1,\ldots,\nu_{i-1}$, and $\delta_1,\ldots,\delta_{i}$ such that 
\[\|\mu_{k_j}-\nu_{j}\|<\delta_j,\quad\quad
    \|m_{xy}\pm \mu_{k_j}\|<1+\delta_j^2/64,\quad\quad
    \|\mu_{k_j}\|> 1-\delta_j^2/64,\]
and $\nu_j\in\conv \big(\Mol_{2\delta_{j+1},\delta_{j}}(M)\big)$ for all $j\in \{1,\ldots,i-1\}$.
Let $k_i\in \mathbb{N}$ be such that \[\|m_{xy}\pm \mu_{k_i}\|<1+\delta_i^2/64\quad\text{ and }\quad \|\mu_{k_i}\|> 1-\delta_i^2/64.\]
By Lemma \ref{lemma1} there exists $\delta_{i+1}>0$ and $\nu_i\in\conv \big(\Mol_{2\delta_{i+1},\delta_{i}}(M)\big)$ such that $\|\mu_{k_i}-\nu_i\|<\delta_i$.

Notice that $\delta_i>2\delta_{i+1}$ for every $i\in\mathbb{N}$. Let $g_{i}\in S_{\Lip_0(M)}$ be such that $g_{i}(\nu_{i})=\|\nu_{i}\|$. Then
\[g_{i}(\mu_{k_i})\ge g_{i}(\nu_{i})-\|\mu_{k_i}-\nu_{i}\|>\|\nu_{i}\|-\delta_i\ge \|\mu_{k_i}\|- \|\mu_{k_i}-\nu_{i}\|-\delta_i>1-3\delta_i\]
and thus
\[\pm g_{i}(m_{xy})\le \|m_{xy}\pm \mu_{k_i}\|-g_{i}(\mu_{k_i})<1+\delta_i-1+3\delta_i=4\delta_i.\]
By Lemma \ref{lemma_something} there exists $f_i\in S_{\Lip_0(M)}$ such that \[f_i(\nu_i)>(\|\nu_i\|-4\delta_i-4\delta_{i+1})/(1+4\delta_i+4\delta_{i+1})\] and
$f_i(p)=0$ for every $p\in M$ with $\pi_2(p)\le2\delta_{i+1}$.

Let $f\colon \bigcup_{i\in\mathbb{N}}\big\{u\in M\colon 2\delta_{i+1}\le \pi_2(u)<\delta_i\big\}\rightarrow \mathbb{R}$ be defined by 
\[f(p)=f_{i}(p)\]
for all $i\in \mathbb{N}$ and $p\in \big\{u\in M\colon 2\delta_{i+1}\le \pi_2(u)<\delta_i\big\}$. Let $u_i=(0,2\delta_{i+1})$ and $v_i=(1,2\delta_{i+1})$.
Then $f_{i}(u_i)=f_i(v_i)=0$ for every $i\in\mathbb{N}$.

Let us show that $f$ is a Lipschitz function. Fix $i,j\in\mathbb{N}$ with $i< j$ and $p\in \big\{u\in M\colon 2\delta_{i+1}\le \pi_2(u)<\delta_i\big\}$ and $q\in \big\{u\in M\colon 2\delta_{j+1}\le \pi_2(u)<\delta_j\big\}$. Then $\delta_j\le\delta_{i+1}$ and therefore
\[d(u_i,u_j)=d(v_i,v_j)=2\delta_{i+1}-2\delta_{j+1}<2\delta_{i+1}\le 4\delta_{i+1}-2\delta_{j}< 2\pi_2(p)-2\pi_2(q)\le 2d(p,q).\]
We have either $u_i\in [p,q]$ or $v_i\in [p,q]$. Assume that $u_i\in [p,q]$ (case $v_i\in [p,q]$ is analogous). Then
\[d(p,u_i)+d(q,u_j)\le d(p,u_i)+d(q,u_i)+d(u_i,u_j)=d(p,q)+d(u_i,u_j)\le 3d(p,q)\]
and thus
\begin{align*}
    |f(p)-f(q)|&\le|f(p)|+|f(q)|\\
    &= |f_{i}(p)-f_{i}(u_i)|+|f_{j}(q)-f_{j}(u_j)|\\
    &\le d(p,u_i)+d(q,u_j)\\
    &\le 3d(p,q).
\end{align*}
Therefore  $f$ is a Lipschitz function.
Now extend $f$ to the whole of $M$ using McShane's extention Theorem. Then 
\[f(\mu_{k_i})\ge f(\nu_{i})-\|\mu_{k_i}-\nu_{i}\|=f_{i}(\nu_{i})-\|\mu_{k_i}-\nu_{i}\|>\frac{\|\nu_i\|-4\delta_i-4\delta_{i+1}}{1+4\delta_i+4\delta_{i+1}}-\delta_i.\]
Since $\|\nu_i\|\rightarrow1$ and $\delta_i\rightarrow0$, we get that $\big(f(\mu_{k_i})\big)$ does not converge to zero. Therefore $\mathcal{F}(M)$ is not WASQ. 

\end{proof}

%\begin{prob}
%(Does there exist/Is there) a property of Banach spaces which characterizes WASQ and LASQ coinciding?
%\end{prob}

\section{Lipschitz-free space over a geodesic metric space}

We now know that LASQ and WASQ are different properties in Lipschitz-free spaces, and thus start looking for a metric characterization of WASQ. In our example from previous section, we strongly relied on the fact that the underlying metric space was not geodesic. Furthermore, the space $\mathcal{F}([0,1])=L_1[0,1]$ is known to be WASQ, so it is natural to consider whether  underlying metric space being geodesic could characterize weakly almost square Lipschitz-free spaces. We provide a partial result that suggests it might be the case, however the full characterization remains as an open problem. 

As a first step, we notice that when $M$ is a geodesic metric space, then we can choose $\varepsilon=0$ in the proof of {\cite[Lemma~2.1]{HKO}} to get the following lemma.

\begin{lem}[c.f. {\cite[Lemma~2.1]{HKO}}]\label{Mol_esitus}
Let $M$ be a geodesic metric space and let $y\in S_{\mathcal{F}(M)}$ with finite support. Then there exist $n\in\mathbb{N}$, $\lambda_i>0$, $u_i,v_i\in M$ with $u_i\neq v_i$, and geodesics $\gamma_{u_iv_i}$ for all $i\in\{1,\ldots,n\}$ such that 
\[y = \sum_{i=1}^n \lambda_i m_{u_i v_i},\quad \sum_{i=1}^n \lambda_i =1,\quad \Lip \gamma_{u_iv_i}=L(\gamma_{u_iv_i})=d(u_i,v_i)\]
and 
\[\image \gamma_{u_iv_i}\cap\image \gamma_{u_jv_j}\subseteq\{u_i,v_i\}\cap\{u_j,v_j\},\]
for all $i,j\in \{1,\ldots,n\}$ with $i\neq j$.
\end{lem}

Now we are ready to present the main result of this section.
\begin{prop}
Let $M$ be a geodesic metric space. Then for every $y\in S_{\mathcal{F}(M)}$ with finite support, there exists a sequence $(\mu_i)\subseteq B_{\mathcal{F}(M)}$ such that $\|y\pm \mu_i\|\rightarrow 1$, $\| \mu_i\|\rightarrow 1$, and $\mu_i\rightarrow0$ weakly.
\end{prop}

\begin{proof}
Fix $y\in S_{\mathcal{F}(M)}$. By Lemma \ref{Mol_esitus}, there exist  $n\in\mathbb{N}$, $\lambda_i>0$, $u_i,v_i\in M$ with $u_i\neq v_i$, and geodesics $\gamma_{u_iv_i}$ for all $i\in\{1,\ldots,n\}$ such that $y = \sum_{i=1}^n \lambda_i m_{u_i v_i}$, $\sum_{i=1}^n \lambda_i =1$, $\Lip \gamma_{u_iv_i}=L(\gamma_{u_iv_i})=d(u_i,v_i)$
and 
\[\image \gamma_{u_iv_i}\cap\image \gamma_{u_jv_j}\subseteq\{u_i,v_i\}\cap\{u_j,v_j\}\]
for all $i,j\in \{1,\ldots,n\}$ with $i\neq j$. To make $M$ into a pointed metric space, choose $\gamma_{u_1,v_1}(1/2)$ as the fixed point $0$.

Define
\[p_{j}^{ik}=
\gamma_{u_iv_i}(j2^{-k})\]
as consecutive points on the path $\gamma_{u_iv_i}$ for all $i\in\{1,\ldots,n\}$, $k\in \mathbb{{N}}$, and $j\in\{0,1,\ldots,2^{k}\}$. We are now ready to define the sequence $\mu_k$ within the support of these points. 

For all $i\in\{1,\ldots,n\}$ and $k\in \mathbb{N}$, set
\[\nu_{k}^i=\sum_{j=1}^{2^{k}}(-1)^{j}\frac{d(p_{j-1}^{ik},p_{j}^{ik})}{d(u_i,v_i)}m_{p_{j-1}^{ik}p_{j}^{ik}}\quad \text{and}\quad \mu_k=\sum_{i=1}^n \lambda_i\nu_{k}^i.\]
We will show that $\|y\pm \mu_k\|\rightarrow1$, $\| \mu_k\|\rightarrow1$ and $ \mu_k\rightarrow0$ weakly. 

For every $i\in \{1,\ldots,n\}$ and $k\in\mathbb{N}$ we have
\[m_{u_iv_i}=\sum_{j=1}^{2^{k}}\frac{d(p_{j-1}^{ik},p_j^{ik})}{d(u_i,v_i)} m_{p_{j-1}^{ik}p_{j}^{ik}}\]
and thus
\begin{equation*}
\|m_{u_iv_i}-\nu_{k}^i\|=2\Big\|\sum_{j=1}^{2^{k-1}}\frac{d(p_{2j-2}^{ik},p_{2j-1}^{ik})}{d(u_i,v_i)} m_{p_{2j-2}^{ik}p_{2j-1}^{ik}}\Big\|\le 2\sum_{j=1}^{2^{k-1}}\frac{d(p_{2j-2}^{ik},p_{2j-1}^{ik})}{d(u_i,v_i)}=1,
\end{equation*}
which gives us $\|y- \mu_k\|\le1$. The proof for $\|y+ \mu_k\|\le1$ is analogous.

%%%%%%%%% -2 tõestus
Next we will show that $\|\mu_k\|\rightarrow1$. It is easy to verify, that for all $k\in \mathbb{N}$ and $i\in \{1,\ldots,n\}$ we have $\|\nu_k^i\|\leq 1$, and thus $\|\mu_k\|\leq 1$ . Therefore it suffices to show that for every $\varepsilon>0$ there exists $l\in\mathbb{N}$ such that $\|\mu_k\|\ge1-8\varepsilon$ for all $k\ge l$ . 

Fix $\varepsilon\in (0,1/8)$ and let 
\[B_i=B\big(u_i,\varepsilon d(u_i,v_i)\big)\cup B\big(v_i,\varepsilon d(u_i,v_i)\big)\]
for every $i\in\{1,\ldots,n\}$.
Note that the sets $\operatorname{Im} \gamma_{u_1,v_1}{\setminus} B_1, \dotsc, \operatorname{Im} \gamma_{u_n,v_n}{\setminus} B_n$ are pairwise disjoint and compact, therefore there exists $l\in \mathbb{N}$ with $2^{-l}<\varepsilon$ such that
\[
\min_{i\neq j\in\{1,\ldots,n\}} d(\operatorname{Im} \gamma_{u_i,v_i}{\setminus} B_i , \operatorname{Im} \gamma_{u_j,v_j}{\setminus} B_j) >\max_{i\in\{1,\ldots,n\}} \frac{d(u_i,v_i)}{2^{l}}.
\]
Fix $k\ge l$ and let $k'\in\mathbb{N}$ be such that 
\[\varepsilon2^{k}\le k'\le \varepsilon2^k+1.\] 
Let 
\[N=\big\{p_{j}^{ik}\colon i\in \{1,\ldots,n\}, j\in \{k',k'+1,\ldots,2^{k}-k'\}\big\}\]
and define $f:N\rightarrow \mathbb{R}$ as
 \[f(p_{j}^{ik})= \begin{cases}
     d(p_{j-1}^{ik},p_{j}^{ik}),&\text{if }j\text{ is odd},\\ 
    0, &\text{ otherwise}. 
 \end{cases}\]
Note that $p_{2^{k-1}}^{1k}=\gamma_{u_1,v_1}(1/2)$ is the fixed point $0$ and thus $f(0)=0$. We will first show that $f\in S_{\Lip_0(N)}$. Considering  $|f(p_{k'}^{1k})-f(p_{k'+1}^{1k})|= d(p_{k'}^{1k},p_{k'+1}^{1k})$, it suffices to show, that $\Lip (f)\leq 1$.
Fix $i,i'\in \{1,\ldots,n\}$ and $j,j'\in\{k',k'+1,\ldots,2^{k}-k'\}$ such that $i\neq i'$ or $j\neq j'$. If $i=i'$, then 
\[\big|f(p_{j}^{ik})-f(p_{j'}^{i'k})\big|\le d(p_{j-1}^{ik},p_{j}^{ik})\le d(p_{j'}^{ik},p_{j}^{ik})=d(p_{j}^{ik},p_{j'}^{i'k}),\]
and if $i\neq i'$, then 
\begin{align*}
    \big|f(p_{j}^{ik})-f(p_{j'}^{i'k})\big|&\le \max\big\{d(p_{j-1}^{ik},p_{j}^{ik}),d(p_{j'-1}^{i'k},p_{j'}^{i'k})\big\}=\max\Big\{ \frac{d(u_i,v_i)}{2^{k}}, \frac{d(u_{i'},v_{i'})}{2^{k}}\Big\}\\
    &\le\max\Big\{ \frac{d(u_i,v_i)}{2^{l}}, \frac{d(u_{i'},v_{i'})}{2^{l}}\Big\}\le d(p_{j}^{ik},p_{j'}^{i'k}),
\end{align*}
since $p_j^{ik}\notin B_i$ and $p_{j'}^{i'k}\notin B_{i'}$. 
Hence $f\in S_{\Lip_0(N)}$, and we can extend $f$ to the whole of $M$ using McShane's extension Theorem.
Then for every $i\in \{1,\ldots,n\}$ we have
\begin{align*}
f(\nu^i_k)&=\sum_{j=1}^{2^{k}}(-1)^{j}\frac{f(p_{j-1}^{ik})-f(p_{j}^{ik})}{d(u_i,v_i)}\\
&\ge\sum_{j=k'+1}^{2^{k}-k'}\frac{d(p_{j-1}^{ik},p_{j}^{ik})}{d(u_i,v_i)}-\sum_{j=1}^{k'}\frac{d(p_{j-1}^{ik},p_{j}^{ik})}{d(u_i,v_i)}-\sum_{j=2^{k}-k'+1}^{2^{k}}\frac{d(p_{j-1}^{ik},p_{j}^{ik})}{d(u_i,v_i)}\\
&=\frac{2^{k}-4k'}{2^{k}}\ge1-4\varepsilon-\frac{4}{2^{k}}>1-8\varepsilon,
\end{align*}
and therefore
\[\|\mu_k\|\ge f(\mu_k)=\sum_{i=1}^n\lambda_if(\nu_k^i)\ge1-8\varepsilon.\]

Lastly we will prove that $\nu_{k}^i\rightarrow0$ weakly for every $i\in\{1,\ldots,n\}$, and thus also $\mu_k\rightarrow0$ weakly. Fix $i\in\{1,\ldots,n\}$ and $f\in \Lip_0(M)$, and define $g\colon [0,1]\rightarrow\mathbb{R}$ as $g=f\gamma_{u_iv_i}$. Then $g$ is Lipschitz and differentiable almost everywhere. Furthermore
\begin{align*}
f(\nu_{k}^i)&=\sum_{j=1}^{2^{k}}(-1)^{j}\frac{d(p_{j-1}^{ik},p_j^{ik})}{d(u_i,v_i)} \frac{f(p_{j-1}^{ik})-f(p_{j}^{ik})}{d(p_{j-1}^{ik},p_{j}^{ik})}\\
&=\frac{1}{d(u_i,v_i)}\sum_{j=1}^{2^{k}}(-1)^{j}\big(g\big((j-1)2^{-k}\big)-g(j2^{-k})\big)\\
&=\frac{1}{d(u_i,v_i)}\sum_{j=1}^{2^{k}}\int\displaylimits_{(j-1)2^{-k}}^{j2^{-k}}(-1)^{j}g'(t)\mathrm{d}t\\
&=\frac{1}{d(u_i,v_i)}\int_0^1 r_k(t)g'(t)\mathrm{d}t,
\end{align*}
where $r_k\colon [0,1]\rightarrow\mathbb{R}$ are the Rademacher functions defined by
\[r_k(t)=\begin{cases}1,& t\in[(j-1)2^{-k},j2^{-k}],\, j\in \{2,4,\ldots,2^k\}\\ -1, & \text{otherwise}.\end{cases}\]
Recall that Rademacher functions are weakly null in $L_1[0,1]$ and $g'\in L_\infty[0,1]$, hence we have $f(\nu_k^i)\rightarrow 0$. This completes the proof.
\end{proof}

\begin{prob}
Is it true that a Lipschitz-free space $\mathcal{F}(M)$ is weakly almost square if and only if the metric space $M$ is geodesic?
\end{prob}

\section{Lipschitz-free space can never have the SSD2P}

According to \cite{MR4183387}, a Banach space $X$ has the \emph{symmetric strong diameter $2d$ property} (briefly, \emph{SSD($2d$)P}) for $d\in (0, 1]$ 
if, for every $\varepsilon>0$ and any number of slices $S_1,\ldots,S_n$ of the unit ball there exist $x_1\in S_1,\dotsc,x_n\in S_n$ and $y\in B_X$ with $\|y\| \geq 1-\varepsilon$ such that for every $i\in\{1,\dotsc,n\}$ we have
\[ x_i\pm dy\in S_i.\]
Note that SSD($2d$)P with $d=1$ means precisely SSD$2$P. By \cite[Propositions~1.6 and 1.7]{MR4183387}, every slice of the unit ball of a Banach space with SSD($2d$)P has at least diameter $2d$, and therefore the unit ball of a Banach space with SSD($2d$)P cannot contain strongly exposed points. Thus, by \cite[Theorem 1.5]{AM}, a complete metric space $M$ has to be length whenever the Lipschitz-free space $\mathcal{F}(M)$ has the SSD($2d$)P.

Let us also point out that if $A$ is a dense set in $B_X$, and $X$ has the SSD($2d$)P, then we can always choose the element $y$ from the definition in the set $A$. Indeed, if $X$ has the SSD($2d$)P, then for every $\varepsilon>0$ and any number of slices $S_1,\ldots,S_n$ of the unit ball there exist $x_1\in S_1,\dotsc,x_n\in S_n$ and $y\in B_X$ with $\|y\| \geq 1-\varepsilon/2$ such that for every $i\in\{1,\dotsc,n\}$ we have
$ x_i\pm dy\in S_i.$ 
For a small enough $\delta>0$ we have $(1-\delta)(x_i\pm dy)\in S_i$ for every $i\in\{1,\dotsc,n\}$.
Hence we may assume $\|y\|<1$ and $\|x_i\pm dy\|<1$ for all $i\in\{1,\ldots,n\}$.
%, by choosing $(1-\delta)x_i$ instead of $x_i$ for small $\delta>0$, if necessary.
Now we can find a sequence $(y_i)$ in $A$ such that $y_i\rightarrow y$. Thus there exists $j\in\N$ such that $\|y_j\|\ge 1-\varepsilon$ and 
$x_i\pm dy_j\in S_i$ for every $i\in\{1,\dotsc,n\}$.

Our main result in this section is the following.

\begin{thm}\label{thm: F(M) never SSD2P}
A Lipschitz-free space cannot have the SSD$2$P. In fact, a Lipschitz-free space cannot have the SSD($2d$)P for any $0<d\leq 1$.
\end{thm}

\begin{proof}
Fix $0<d\leq 1$. Let $\varepsilon\in(0,d/6)$ and let $n\in \mathbb{N}$ be such that $4/n\leq \varepsilon$. We assume by contradiction, that $\mathcal{F}(M)$ has SSD($2d$)P. Thus we can assume, that $M$ is length and $\mathcal{F}(M)$ is infinite dimensional. Therefore $M$ is infinite and we can choose pairwise distinct points $p_1,\ldots,p_n\in M\setminus\{0\}$. Let
\[R=\frac{1}{4}\min\big\{d(u,v)\colon u,v\in \{p_1,\ldots,p_n,0\}, u\neq v\big\}\]
and let $r=\varepsilon R/2$.
For every $i\in\{1,\dotsc,n\}$, set $B_i = B(p_i, R)$ and $C_i=B(p_i,r)$. Then $B_1,\dotsc,B_n$ are pairwise disjoint. For each $i\in \{1,\ldots,n\}$, we define functions $f_i\in S_{\Lip_0(M)}$ as
\[f_i(u)=\begin{cases}r-d(u,p_i),& u\in C_i,\\ 0, & u\in M\setminus C_i,\end{cases}\]
and slices $S_i\subset B_{\mathcal{F}(M)}$ as $S_i=S(f_i,\varepsilon)$. Since $M$ is length, then $p_i$ is a cluster point and thus $\|f_i\|=1$ for every $i\in\{1,\ldots,n\}$. Therefore our slices are defined correctly. Let $\mu\in B_{\mathcal{F}(M)}$ be such that $\|\mu\|\ge 1-\varepsilon$. Our goal is to show that for some $i\in \{1,\ldots,n\}$ we have $\nu+d\mu\notin S_i$ for every $\nu\in S_i$, from which it immediately follows that $\mathcal{F}(M)$ does not have the SSD($2d$)P. Since the set of finite convex combinations of molecules is dense in $\mathcal{F}(M)$, it suffices to consider the case $\mu = \sum_{j=1}^m \frac{1}{m}\, m_{u_j v_j}$. By a cardinality argument, we fix an $i\in \{1,\dotsc, n\}$ such that the set
\[
J = \{j\in \{1,\dotsc,m\}\colon u_j\in B_i \text{ or } v_j\in B_i\}
\]
has at most $2m/n$ elements. 

Fix $\nu\in S_i$. We are going to show that $\|\nu+d\mu\|>1$, which implies $\nu+d\mu\not\in S_i$. We can find $\nu_0=\sum_{j=1}^k\frac1k m_{x_jy_j}\in S_i$ such that $\|\nu-\nu_0\|<2\varepsilon$ and $x_j,y_j\in C_i$ for all $j\in \{1,\ldots,k\}$. %Küsimus: kas võib kirjutada: We can find $z\in S_i$ with finite support such that $\|z-x\|<2\varepsilon$ and $\supp(z)\in C_i$?? See ei ole samaväärne, kuna $0$ ei kuulu supporti aga võib kuuluda kombinatsiooni. 

Next, we will define a specific $g\in \Lip_0(M)$. Let $f\in B_{\Lip_0(M)}$ be such that $f(\mu)=(1-\varepsilon)\|\mu\|$ and $\|f\|=1-\varepsilon$. We set $g|_{M{\setminus} B_i}=f|_{M{\setminus} B_i}$, $g|_{C_i}=f_i|_{C_i}+f(p_i)-f_i(p_i)$, and by McShane's extension Theorem, extend $g$ to be defined in the whole of $M$ while preserving the Lipschitz constant.
Then $\|g\|\leq 1$ because, for every $a\in M\setminus B_i$ and $b\in C_i$, one has
\begin{align*}
    |g(a)-g(b)|&=|f(a)-f_i(b)-f(p_i)+f_i(p_i)|\leq (1-\varepsilon)d(a,p_i)+d(b,p_i)\\
    &\leq d(a,b)+2d(b,p_i)-\varepsilon d(a,p_i)\leq d(a,b)+2r-\varepsilon R= d(a,b).
\end{align*}
Note that
\[(f-g)(\mu) = (f-g)\bigl(\sum_{j\in J}\frac{1}{m}\, m_{u_jv_j}\bigr) \leq (\|f\|+\|g\|) \sum_{j\in J}\frac{1}{m}\|m_{u_j v_j}\|\leq 2\frac{2m}{mn} = \frac{4}{n}.\]
Therefore
\begin{align*}
    dg(\mu) &= df(\mu)-d(f-g)(\mu)\geq d(1-\varepsilon)\|\mu\|-\frac{4d}{n}\\
    &\geq d\|\mu\|-\varepsilon  -\frac{4}{n}\geq d(1-\varepsilon)-2\varepsilon\geq d-3\varepsilon,
\end{align*}
hence
\begin{align*}
    \|\nu_0+d\mu\|
    \geq g(\nu_0)+dg(\mu)\geq  1-\varepsilon+d-3\varepsilon= 1+d-4\varepsilon
\end{align*}
and 
\[\|\nu+d\mu\|\geq \|\nu_0+d\mu\|-\|\nu-\nu_0\|\geq 1+ d -6\varepsilon>1.\]
\end{proof}

\begin{rem}
Recall that a Banach space $X$ has the SSD$2$P if and only if $X^{**}$ has the weak$^*$ SSD$2$P (see \cite[Section~5]{HLLN}). Using a similar argument, one can prove that for SSD$(2d)$P as well. Thus it immedietly follows from Theorem \ref{thm: F(M) never SSD2P} that $\Lip_0(M)^*$ can not have the weak$^*$ SSD($2d$)P for any $0<d\leq 1$.
\end{rem}

\section*{Acknowledgements}
The authors are grateful to Rainis Haller for several helpful comments concerning the paper. The authors also wish to express their gratitude to the referee for pointing out a mistake in Section 3 and for several helpful comments regarding other sections.

This work was supported by the Estonian Research
Council grant (PRG1901).

%\addcontentsline{toc}{section}{References}
%\bibliography{bibliography}{}
%\bibliographystyle{amsplain}

\bibliographystyle{amsplain}

\end{document}